\documentclass[submission]{FPSAC2022}
\usepackage{pgf,tikz}
\usetikzlibrary{arrows}

\newcommand{\bn}{\ensuremath{\mathbf n}}
\newcommand{\R}{\ensuremath{\mathbb R}}
\newcommand{\Z}{\ensuremath\mathbb{Z}}
\newcommand{\W}{\ensuremath{\mathcal{W}}}
\DeclareMathOperator{\Id}{Id}

\newtheorem{theorem}{Theorem}[section]
\newtheorem{cor}[theorem]{Corollary}

\newtheorem{proposition}[theorem]{Proposition}
\newtheorem{defn}[theorem]{Definition}
\newtheorem{remark}[theorem]{Remark}
\newtheorem{example}[theorem]{Example}

\title{Blowup polynomials and delta-matroids of graphs}

\author{Projesh Nath
Choudhury\thanks{\href{mailto:projeshc@iisc.ac.in}{projeshc@iisc.ac.in}.
P.N.~Choudhury was supported by National Post-Doctoral Fellowship (NPDF)
PDF/2019/000275 from SERB (Govt.~of India) and by a C.V.\ Raman
Postdoctoral Fellowship (IISc).}\addressmark{1} and
Apoorva Khare\thanks{\href{mailto:khare@iisc.ac.in}{khare@iisc.ac.in}.
A.~Khare was partially supported by
Ramanujan Fellowship grant SB/S2/RJN-121/2017,
MATRICS grant MTR/2017/000295, and
SwarnaJayanti Fellowship grants SB/SJF/2019-20/14 and DST/SJF/MS/2019/3
from SERB and DST (Govt.~of India),
and by grant F.510/25/CAS-II/2018(SAP-I) from UGC (Govt.~of
India).}\addressmark{1}\addressmark{2}}

\address{\addressmark{1}Department of Mathematics, Indian Institute of
Science, Bangalore -- 560012, India \\
\addressmark{2}Analysis and Probability Research Group, Bangalore --
560012, India}

\received{\today}

\abstract{For every finite simple connected graph $G = (V,E)$, we
introduce an invariant, its blowup-polynomial $p_G(\{ n_v : v \in V \})$.
This is obtained by dividing the determinant of the distance matrix of
its blowup graph $G[\bn]$ (containing $n_v$ copies of $v$) by an
exponential factor. We show that $p_G(\bn)$ is indeed a polynomial
function in the sizes $n_v$, which is moreover multi-affine and
real-stable. This associates a hitherto unexplored delta-matroid to each
graph $G$; and we provide a second unexplored one for each tree. As
another consequence, we obtain a new characterization of complete
multipartite graphs, via the homogenization at $-1$ of $p_G$ being
completely/strongly log-concave, i.e., Lorentzian. (These results extend
to weighted graphs.) Finally, we show $p_G$ is indeed a graph invariant,
i.e., $p_G$ and its symmetries (in the variables $\bn$) recover $G$ and
its isometries, respectively.}

\keywords{distance matrix, blowup-polynomial, real-stable polynomial,
Zariski density, delta-matroid}

\begin{document}

\maketitle

Fifty years ago, Graham and Pollak~\cite{Graham-Pollak} showed the
following striking result in algebraic combinatorics: \textit{Given a
tree $T = (V,E)$ with distance matrix $D_T$, the scalar $\det D_T$ is
independent of the tree structure, and depends only on $|V| = |E|+1$.}
Here, $D_G$ for a finite connected, simple graph $G$ denotes its distance
matrix, with the $(v,w)$ entry given by the length of the shortest path
connecting $v \neq w \in V$, and $(D_G)_{vv} = 0\ \forall v \in V$.
This result has been extended to multiple other settings, including
$q$-distance matrices, multiplicative distances, and even combinations of
these -- see e.g.~\cite{CK-tree} and its references for details and for
an overarching generalization. The area has remained active ever since.

Graham then explored the spectral side with Lov\'asz \cite{GL}, including
computing the characteristic polynomial (and roots) and inverse of $D_T$.
This line of research too remains active, and has led to the study of
``distance spectra'' of graphs -- see e.g.~the survey \cite{AH}.

Our work was motivated by both directions. On the algebraic side, we
sought natural graph families $\{ G_i : i \in I \}$ -- e.g.\ trees on $n$
vertices -- such that the map $i \mapsto \det D_{G_i}$ is a ``nice''
function from $I$ to $\R$. On the analysis side, it is well-known that
the characteristic polynomial $\det(x \Id - D_G)$ of the distance matrix
of $G$ does not recover $G$, i.e., there are graphs $G \not\cong H$ with
the same number of nodes, which are ``distance co-spectral''. Thus, we
were interested in finding a different byproduct of $D_G$ that recovers
$G$.

The purpose of this note is to describe such a byproduct of $D_G$ (or of
$G$), which we introduce in the work~\cite{CK-blowup}, and which we term
the \textit{(multivariate) blowup-polynomial} of $G$. We then explain how
this polynomial achieves the above two goals. A third, interesting
byproduct of our work is a -- to our knowledge -- novel family of
delta-matroids, one for every graph $G$ (and we introduce a second novel
delta-matroid for every tree). This third holds because the
blowup-polynomial turns out to be multi-affine and real-stable.

\section{The blowup-polynomial of a graph, and its symmetries}

We begin by introducing the key ingredient needed to define the
blowup-polynomial: the family of \textit{blowup graphs} of $G$:

\begin{defn}
Given a finite simple connected (unweighted) graph $G = (V_G,E_G)$, and a
set of positive integers $\bn = \{ n_v : v \in V_G \}$, the \emph{blowup
graph} $G[\bn]$ is the finite simple connected graph with $n_v$ copies of
the vertex $v$, such that a copy of $v$ is adjacent to one of $w$ if and
only if $v \neq w$ and $(v,w) \in E_G$. Define $M_G := D_G + 2
\Id_{V_G}$, where $D_G$ is the distance matrix of $G$.
\end{defn}

\noindent Blowup graphs are studied in e.g.~\cite{HHN,KKOT,KSS} in
extremal and probabilistic graph theory.

We now claim that -- akin to trees on $n$ vertices for any fixed $n \geq
1$ -- the family of blowups of a fixed graph $G$ is well-behaved
vis-a-vis computing $\det D_{G[\bn]}$:

\begin{theorem}\label{Tmetricmatrix}
Given a finite simple connected (unweighted) graph $G$, there exists a
polynomial $p_G(\bn)$ in the {\em sizes} $n_v$, with integer
coefficients, such that
\[
\det D_{G[\bn]} = (-2)^{\sum_v (n_v-1)} p_G(\bn), \qquad \bn \in
\Z_{>0}^V.
\]
Also, $p_G$ is multi-affine in $\bn$, with constant term $(-2)^{|V|}$
and linear term $-(-2)^{|V|} \sum_{v \in V} n_v$.
\end{theorem}

Here and below, we mildly abuse notation and refer to both the integer
sizes as well as indeterminates by $n_v$; this will be clear from
context. Also recall, a polynomial $p(\{ n_v \})$ is multi-affine if
$\deg_{n_v}(p) \leq 1$ for all $v$.

\begin{defn}
For a graph $G$ as in Theorem~\ref{Tmetricmatrix}, define its
\emph{(multivariate) blowup-polynomial} to be $p_G(\bn) \in \Z[\bn]$,
where we think of the $n_v$ as indeterminates. Also define the
\emph{univariate blowup-polynomial} of $G$ to be $u_G(n) :=
p_G(n,n,\dots,n)$.
\end{defn}

We clarify this definition with a remark. The polynomial function (by
Theorem~\ref{Tmetricmatrix})
\[
\bn \mapsto  (-2)^{-\sum_v (n_v-1)} \det D_{G[\bn]}, \qquad \bn \in
\Z_{>0}^V
\]
has to first be extended to $\R^V$ from its Zariski dense subset
$\Z_{>0}^V$. It can then be identified with a polynomial in $\R[\bn]$
(with integer coefficients), and it is this polynomial that we denote
here and below by $p_G(\bn)$ as well.

\begin{proof}[Proof of Theorem~\ref{Tmetricmatrix}]
We provide a quick sketch; the key ingredient is again algebraic here:
Zariski density. (In fact, this result holds over a general commutative
ring, and we refer the reader to the full paper~\cite{CK-blowup} for
details.) Let $k := |V|$, \textbf{fix} (throughout this note) an
enumeration $(n_1, \dots, n_k)$ of $\{n_v : v \in V\}$, let $D_G =
(d_{ij})_{i,j=1}^k$, and define
\[
K := \sum_{i=1}^k n_i, \qquad
\W_{K \times k} := \begin{pmatrix} {\bf 1}_{n_1 \times 1} & 0_{n_1
\times 1} & \cdots & 0_{n_1 \times 1} \\
0_{n_2 \times 1} & {\bf 1}_{n_2 \times 1} & \cdots & 0_{n_2 \times 1} \\
\vdots & \vdots & \ddots & \vdots \\
0_{n_k \times 1} & 0_{n_k \times 1} & \cdots & {\bf 1}_{n_k \times 1}
\end{pmatrix}.
\]

Given an integer tuple $\bn \in \Z_{>0}^k$, recall that $D_{G[\bn]} =
M_{G[\bn]} - 2 \Id_K$. Notice that $M_{G[\bn]}$ is a block $k \times k$
matrix with $(i,j)$ block $d_{ij} \cdot {\bf 1}_{n_i \times n_j}$ for $i
\neq j$ and $2 \cdot {\bf 1}_{n_i \times n_i}$ for $i=j$; in particular,
$M_{G[\bn]} = \W M_G \W^T$. We now employ Zariski density, by first
considering the entries of $M_G$ as well as the sizes $n_i$ to be
variables, and working over the field $\mathbb{F}$ of rational functions
in these, with coefficients in $\mathbb{Q}$. In particular, $\det M_G \in
\mathbb{F}^\times$. We compute, using Schur complements repeatedly:
\begin{align}\label{Ecomp}
\det D_{G[\bn]} = &\ \det (\W M_G \W^T - 2 \Id_K) = \det 
\begin{pmatrix} -2 \Id_K & - \W \\ \W^T & M_G^{-1} \end{pmatrix}
\det(M_G),\\
= &\ (-2)^K \det (M_G^{-1} - 2^{-1} \W^T \W) \det (M_G) = (-2)^{K-k}
\det((-2)\Id_k + \Delta_\bn M_G), \notag
\end{align}
where $\Delta_\bn = \W^T \W$ is the diagonal matrix with $(i,i)$ entry $n_i$.
Now~\eqref{Ecomp} proves the result over the field $\mathbb{F}$ of
rational functions, hence -- by Zariski density -- in the subring of
polynomials in the same variables, since both sides of~\eqref{Ecomp} are
polynomial \textit{functions}. As $\mathbb{Q}$ is infinite, we obtain an
equality of polynomials, both of which have integer coefficients.
Finally, specialize the sizes $n_i$ and the entries of $M_G$ to the given
graph-data.
\end{proof}

\begin{remark}\label{Rformula}
It also follows from the above proof that $p_G(\bn) = \det (\Delta_\bn
M_G - 2 \Id_k)$.
\end{remark}

Theorem~\ref{Tmetricmatrix} and its proof enable us to do more: we can
compute the coefficient of every monomial in $p_G$, and relate $p_G$ to
$p_H$ for certain induced subgraphs $H$ of $G$:

\begin{proposition}\label{Pcoeff}
Notation as above.
\begin{enumerate}
\item Given a subset $I \subset V$, the coefficient in $p_G(\bn)$ of
$\prod_{i \in I} n_i$ is
$(-2)^{|V \setminus I|} \det (M_G)_{I \times I}$,
where $(M_G)_{I \times I}$ is the principal submatrix of $M_G$ formed by
the rows and columns indexed by~$I$.

\item Let $H$ be an induced subgraph of $G$ with vertex set $I \subset V$
and no isolated nodes. Then,
\[
p_H(\{ n_i : i \in I \}) = p_G(\bn)|_{n_j = 0\; \forall j \not\in
I} \cdot (-2)^{-|V \setminus I|}.
\]
Thus if some monomial $\prod_{i \in I_0} n_i$ (for $I_0 \subset I$) does
not occur in $p_H$, it does not occur in $p_G$.

\item Suppose $H,K$ are induced subgraphs of $G$, say on node sets $I,J
\subset V$ respectively, and each without isolated nodes.
If $H,K$ are isomorphic, then the coefficients in $p_G(\bn)$ of $\prod_{i
\in I} n_i$ and $\prod_{j \in J} n_j$ are equal.

\item The iterated blowup of a graph $G = (V,E)$ is also a blowup of $G$.
In particular, the blowup-polynomial of $p_{G[\bn]}$ has total degree at
most $|V|$, for all $\bn \in \Z_{>0}^V$.
\end{enumerate}
\end{proposition}

As a simple illustration of the final assertion here, notice that the
path graph $P_3$, the cycle $C_4$, and all star graphs $K_{1,n}$ are
instances of complete bipartite graphs $K_{r,s}$. As $K_{r,s} =
K_2[(r,s)]$ is a blowup of the edge $K_2$, the blowup-polynomials of all
of these graphs are multi-affine of degree $2$, and can be easily
computed.

Proposition~\ref{Pcoeff} has multiple applications; we provide two here.
First, it makes tractable the computation of $p_G(\cdot)$ for certain
more involved graphs. Here is an example.

\begin{example}
Given integers $k,l$ with $0 \leq l \leq k-2$, let $K_k^{(l)}$ denote the
graph on vertices $\{ 1, \dots, k \}$, with all edges connected except
for $(1, 2), \dots, (1,l+1)$. These form a family of chordal graphs, with
isomorphism/isometry group $S_l \times S_{k-l-1}$ corresponding to the
partition of the vertex set $V = \{ 1 \} \sqcup \{ 2, \dots, l+1 \}
\sqcup \{ l+2, \dots, k \}$. Now we have:
\begin{align}\label{EKkl}
p_{K_k^{(l)}}(\bn) = &\ \sum_{r=0}^l \sum_{s=0}^{k-l-1} \left[
(-2)^{k-r-s} (1 + r + s) \right] e_r(n_2, \dots, n_{l+1}) e_s(n_{l+2},
\dots, n_k)\\
&\ + n_1 \sum_{r=0}^l \sum_{s=0}^{k-l-1} \left[ (-2)^{k-r-s-1} (1 -r)
(s+2) \right] e_r(n_2, \dots, n_{l+1}) e_s(n_{l+2}, \dots, n_k), \notag
\end{align}
with $e_r(\cdot)$ the elementary symmetric polynomial.
(The graphs $K_k^{(1)}$ were crucially used in \cite{GKR-critG}.)
\end{example}

The above decomposition of the nodes of $K_k^{(l)}$ is into subsets, each
containing nodes that are all isomorphic to one another. These
auto-isometries (i.e., adjacency-preserving bijections) of the underlying
graph translate into \textit{symmetries} of the blowup-polynomial, as
seen in~\eqref{EKkl}. (We may thus call $p_G$ a \textit{partially
symmetric polynomial}.) Conversely, it is natural to ask if $p_G$ can
recover the auto-isometries of $G$ -- and more strongly, if $p_G$
recovers the graph $G$ itself. Our next result provides a positive
answer.

\begin{proposition}
Given $G$ as above, the symmetries of $p_G$ coincide with the
auto-isometries of $G$. More strongly, the polynomial $p_G$ recovers $G$.
However, this is not true for the univariate specialization $u_G$.
\end{proposition}

\begin{proof}[Proof-sketch]
The first claim follows from the second, which holds because the
Hessian equals
\[
\mathcal{H}(p_G) := ((\partial_{n_i} \partial_{n_j} p_G)({\bf
0}))_{i,j=1}^k = (-2)^k {\bf 1}_{k \times k} - (-2)^{k-2} M_G^{\circ 2},
\]
where given a matrix $M = (m_{ij})$, $M^{\circ 2} := (m_{ij}^2)$ is its
entrywise square. Finally, to study $u_G$, define the graphs $H,K$ in
Figure~\ref{Fig1}, both with vertices $\{ 1, \dots, 6 \}$.
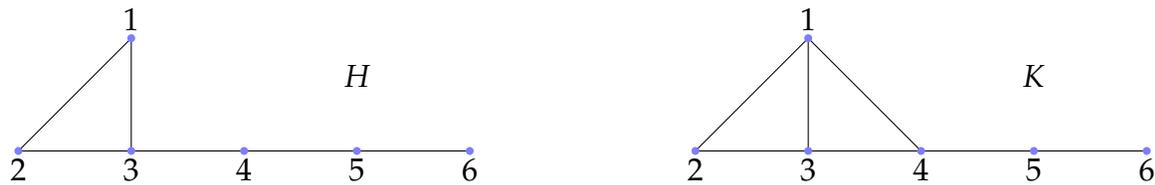
\begin{figure}[ht]
\definecolor{xdxdff}{rgb}{0.49,0.49,1}
\begin{tikzpicture}
\draw (0,1.5)-- (-1.5,0);
\draw (0,0)-- (0,1.5);
\draw (-1.5,0)-- (0,0);
\draw (0,0)-- (1.5,0);
\draw (1.5,0)-- (3,0);
\draw (3,0)-- (4.5,0);
\fill [color=xdxdff] (-1.5,0) circle (1.5pt);
\draw[color=black] (-1.5,-0.25) node {$2$};
\fill [color=xdxdff] (0,0) circle (1.5pt);
\draw[color=black] (0,-0.25) node {$3$};
\fill [color=xdxdff] (1.5,0) circle (1.5pt);
\draw[color=black] (1.5,-0.25) node {$4$};
\fill [color=xdxdff] (3,0) circle (1.5pt);
\draw[color=black] (3,-0.25) node {$5$};
\fill [color=xdxdff] (4.5,0) circle (1.5pt);
\draw[color=black] (4.5,-0.25) node {$6$};
\fill [color=xdxdff] (0,1.5) circle (1.5pt);
\draw[color=black] (0,1.75) node {$1$};
\draw[color=black] (3,1) node {$H$};
\draw (9,1.5)-- (7.5,0);
\draw (9,1.5)-- (9,0);
\draw (9,1.5)-- (10.5,0);
\draw (7.5,0)-- (9,0);
\draw (9,0)-- (10.5,0);
\draw (10.5,0)-- (12,0);
\draw (12,0)-- (13.5,0);
\fill [color=xdxdff] (7.5,0) circle (1.5pt);
\draw[color=black] (7.5,-0.25) node {$2$};
\fill [color=xdxdff] (9,0) circle (1.5pt);
\draw[color=black] (9,-0.25) node {$3$};
\fill [color=xdxdff] (10.5,0) circle (1.5pt);
\draw[color=black] (10.5,-0.25) node {$4$};
\fill [color=xdxdff] (12,0) circle (1.5pt);
\draw[color=black] (12,-0.25) node {$5$};
\fill [color=xdxdff] (13.5,0) circle (1.5pt);
\draw[color=black] (13.5,-0.25) node {$6$};
\fill [color=xdxdff] (9,1.5) circle (1.5pt);
\draw[color=black] (9,1.75) node {$1$};
\draw[color=black] (12,1) node {$K$};
\end{tikzpicture}
\caption{Two non-isomorphic graphs on six vertices with co-spectral
blowups}\label{Fig1}
\end{figure}
Next, we define:
\[
H' := H[(2,1,1,2,1,1)], \qquad K' := K[(2,1,1,1,1,2)].
\]
It is easily checked by direct computations that $H', K'$ are not
isomorphic, but
\[
u_{H'}(n) = u_{K'}(n) = -320 n^6 + 3712 n^5 - 10816 n^4 + 10880 n^3 -
1664 n^2 - 2048 n + 256. \qedhere
\]
\end{proof}

Thus, $H' \not\cong K'$ (both with $|V| = 8$) are graphs whose distance
matrices have the same characteristic polynomial and equal univariate
polynomials $u_{H'} = u_{K'}$; but $p_{H'} \neq p_{K'}$.

Our second application of Proposition~\ref{Pcoeff} involves a special
case of the graphs $K_k^{(l)}$ -- namely, for $l=0$, in which case
$K_k^{(l)} = K_k$, a complete graph. In this case, one checks:
\begin{equation}\label{EPKk}
p_{K_k}(n_1, \dots, n_k) = \prod_{i=1}^k (n_i - 2) + \sum_{i=1}^k n_i
\prod_{i' \neq i} (n_{i'} - 2).
\end{equation}
This is ``fully'' symmetric in the $n_i$. In fact, there are no other
graphs with this property:

\begin{proposition}\label{Psymm}
Given a graph $G$ as above, the blowup-polynomial $p_G(\bn)$ is symmetric
in the variables $\{ n_i : 1 \leq i \leq k \}$ if and only if $G$ is
complete.
\end{proposition}

\section{Real-stability and related properties}

Our next goal is to explain how the blowup-polynomial gives rise to a
hitherto unexplored delta-matroid for every graph. (More generally, one
obtains such a delta-matroid from every finite metric space -- see
Remark~\ref{Rmetric}.) This will follow from the polynomial $p_G$
possessing additional desirable features, which we describe in this
section.

As a motivating example, note that specializing the polynomial
$p_{K_k}(\bn)$ in~\eqref{EPKk} yields the univariate blowup-polynomial
$u_{K_k}(n) = (n-2)^{k-1} (kn + n-2)$, and this is real-rooted. More
generally, this turns out to hold for all graphs $G$ -- in fact, far more
is true. Real-rootedness is the one-variable manifestation of a more
general, and far more powerful notion: a polynomial $p(z_1, \dots, z_k)$
with real coefficients and complex arguments is said to be
\textit{real-stable} if it is non-vanishing whenever $\Im(z_j) > 0\
\forall j$. Real-stable polynomials and their generalizations are the
focus of tremendous recent research, see e.g.\ the well-known papers by
Borcea--Br\"and\'en \cite{BB1,BB2,BB3} and Marcus--Spielman--Srivastava
\cite{MSS1,MSS2}, in which longstanding conjectures of Bilu--Linial,
Johnson, Kadison--Singer, Lubotzky, and others are resolved, and the
Laguerre--P\'olya--Schur program from the early 20th century is
significantly extended (among other remarkable achievements).

In combinatorics, the importance of real-rootedness and (strong)
log-concavity is very well known, see e.g.~\cite{Brenti}, \cite{Stanley}.
Recently, there has been much work in going beyond these notions and
studying the connections of stability to combinatorics and statistical
physics; see e.g.~\cite{Branden-survey}, \cite{Pemantle}. Our next result
shows that graph blowup-polynomials $p_G(\cdot)$ are indeed real-stable
(which is what will yield novel delta-matroids, below):

\begin{theorem}\label{Tstable}
Given a finite simple connected graph $G$, its blowup-polynomial
$p_G(\bn)$ is real-stable in the variables $\{ n_v : v \in V \} = \{ n_1,
\dots, n_k \}$. (In particular, $u_G(n)$ is always real-rooted.)
\end{theorem}

Recall from \cite{Branden}, \cite{WW} that a multi-affine polynomial
$f(z_1, \dots, z_n)$ is real-stable if and only if $\partial_i f \cdot
\partial_j f \geq f \cdot \partial_i \partial_j f$ on $\R^n$, for all
$i,j$. The class of real-stable multi-affine polynomials is also
connected to matroids; see \cite{Branden}, \cite{COSW}.
Theorem~\ref{Tstable} says that graph blowup-polynomials $p_G(\bn)$
provide novel (to our knowledge) examples of such maps.

\begin{proof}
As the goal is to prove real-stability, in this proof we write $p_G(z_1,
\dots, z_k)$ to indicate that the variables are complex (rather than
algebraic). From Remark~\ref{Rformula},
\begin{align}\label{Estable}
p_G({\bf z}) = &\ \det (\Delta_{\bf z} M_G - 2 \Id_k) =
\prod_{j=1}^k z_j \cdot \det(2^{-1}M_G - \Delta_{\bf z}^{-1})
\cdot 2^k \notag\\
= &\ 2^k \prod_{j=1}^k z_j \cdot \det \left( 2^{-1}M_G +
\sum_{j=1}^k (-z_j^{-1} E_{jj}) \right),
\end{align}
where $E_{jj} \in \Z^{k \times k}$ is the elementary matrix with
$(j,j)$-entry $1$. Now we recall a fundamental determinantal example of
real-stable polynomials by Borcea--Br\"and\'en -- see \cite{BB1} (or
\cite[Lemma 4.1]{Branden}). The authors show that if $A_1, \dots, A_k, B$
are real symmetric matrices, with all $A_j$ positive semidefinite, then
the polynomial
\begin{equation}\label{EBH}
f(z_1, \dots, z_k) := \det \left( B + \sum_{j=1}^k z_j A_j \right)
\end{equation}
is real-stable or identically zero. Moreover, ``inversion preserves
stability'': if $g(z_1, \dots, z_k)$ is a polynomial with $z_j$-degree
$d_j \geq 1$ that is real-stable, then so is $z_1^{d_1} g(-z_1^{-1}, z_2,
\dots, z_k)$. (This is because the map $z \mapsto -1/z$ preserves the
upper half-plane.) Now apply~\eqref{EBH} to $A_j = E_{jj}$ and $B =
2^{-1} M_G$, and then apply inversion in each variable, to conclude
via~\eqref{Estable} that $p_G$ is real-stable.
\end{proof}

Returning to $u_G$, which we now know is real-rooted, we also note that
it is indeed related to the distance spectrum of $G$ (i.e., to the
characteristic polynomial of $D_G$):

\begin{proposition}
For any finite simple connected (unweighted) graph $G$, a real number $n$
is a root of $u_G$ if and only if $n \neq 0$ and $2 n^{-1} - 2$ is an
eigenvalue of $D_G$ (with the same multiplicity).
\end{proposition}

\subsection{A novel characterization of complete multipartite graphs}

We next mention two other notions related to stability, which have been
greatly studied in recent years, and which are not satisfied by $p_G$. By
the final assertion in Theorem~\ref{Tmetricmatrix}, the coefficients of
the multi-affine polynomial $p_G$ cannot be normalized to form a
probability distribution, since they are not all of the same sign.
Similarly, the polynomial $p_G$ is clearly not homogeneous. In two
fundamental and important papers, stable polynomials with these two
properties have been studied (in broader settings) by
Borcea--Br\"and\'en--Liggett \cite{BBL} and Br\"and\'en--Huh \cite{BH},
under the name of \textit{strongly Rayleigh measures/polynomials} and
\textit{Lorentzian polynomials}, respectively. Our next result explains
that while $p_G$ is neither strongly Rayleigh nor Lorentzian, a suitable
normalization/homogenization can be. In fact, we completely characterize
all such graphs:

\begin{theorem}\label{Tlorentz}
Given a graph $G$ as above, define its {\em homogenized
blowup-polynomial}
\begin{equation}
\widetilde{p}_G(z_0, z_1, \dots, z_k) := (-z_0)^k p_G \left(
\frac{z_1}{-z_0}, \dots, \frac{z_k}{-z_0} \right) \in \mathbb{R}[z_0,
z_1, \dots, z_k].
\end{equation}
The following are equivalent.
\begin{enumerate}
\item The homogenized polynomial $\widetilde{p}_G(z_0, z_1, \dots, z_k)$
is real-stable.

\item The polynomial $\widetilde{p}_G(z_0, z_1, \dots, z_k)$ is
{\em Lorentzian}. That is, $\widetilde{p}_G(\cdot)$ is homogeneous
of degree $k$ with non-negative coefficients, and given indices $0 \leq
j_1, \dots, j_{k-2} \leq k$, if
\[
g(z_0, z_1, \dots, z_k) := \left( \partial_{z_{j_1}} \cdots
\partial_{z_{j_{k-2}} } \widetilde{p}_G \right)(z_0, z_1, \dots, z_k),
\]
then its Hessian matrix $\mathcal{H}_g := (\partial_{z_i} \partial_{z_j}
g)_{i,j=0}^k \in \R^{(k+1) \times (k+1)}$ is Lorentzian (i.e.,
$\mathcal{H}_g$ is nonsingular and has exactly one positive eigenvalue).

\item $\widetilde{p}_G(\cdot)$ has all coefficients non-negative (i.e.,
of the monomials $z_0^{k - |J|} \prod_{j \in J} z_j$).

\item $(-1)^k p_G(-1,\dots,-1) > 0$, and the normalized ``reflected''
polynomial
\[
(z_1, \dots, z_k) \quad \mapsto \quad \frac{p_G(-z_1, \dots,
-z_k)}{p_G(-1,\dots,-1)}
\]
is {\em strongly Rayleigh}. That is, this multi-affine polynomial is
real-stable, has non-negative coefficients (of all monomials
$\prod_{j \in J} z_j$), and these sum up to $1$.

\item The matrix $M_G = D_G + 2 \Id_k$ is positive semidefinite.

\item The graph $G$ is a blowup of a complete graph -- that is, $G$ is a
complete multipartite graph.
\end{enumerate}
\end{theorem}

Theorem~\ref{Tlorentz} characterizes the complete multipartite graphs in
terms of stability. We refer the reader to the full
paper~\cite{CK-blowup} for the proof.

It turns out that two additional, related notions in the literature also
characterize the complete multipartite graphs, and we mention them here
for completeness. Suppose a polynomial $p \in \R[z_1, \dots, z_k]$ has
non-negative coefficients. In~\cite{Gurvits}, Gurvits defines $p$ to be
\textit{strongly log-concave} if for every $\alpha \in \mathbb{Z}_{\geq
0}^k$, either the derivative $\displaystyle
\partial^\alpha (p) := \prod_{i=1}^k \partial_{x_i}^{\alpha_i} \cdot p$
is identically zero, or $\partial^\alpha p > 0$ and $\log(\partial^\alpha
(p))$ is concave on $(0,\infty)^k$. Next in~\cite{AGV}, Anari, Oveis
Gharan, and Vinzant define $p$ to be \textit{completely log-concave} if for
all $m \in \Z_{>0}$ and matrices $A = (a_{ij}) \in [0,\infty)^{m \times
k}$, either the derivative
$\displaystyle
\partial_A (p) := \prod_{i=1}^m \left( \sum_{j=1}^k a_{ij} \partial_{x_j}
\right) \cdot p$
is identically zero, or $\partial_A(p)> 0$ and $\log(\partial_A (p))$ is
concave on $(0,\infty)^k$. We now have:

\begin{cor}\label{Clorentz}
Notation as in Theorem~\ref{Tlorentz}.
Then $G$ is complete multipartite if and only if either of the following
holds:
\begin{enumerate}
\setcounter{enumi}{6}
\item The polynomial $\widetilde{p}_G(z_0, \dots, z_k)$ is strongly
log-concave.

\item The polynomial $\widetilde{p}_G(z_0, \dots, z_k)$ is completely
log-concave.
\end{enumerate} 
\end{cor}

\begin{proof}
For arbitrary real homogeneous polynomials, \cite[Theorem 2.30]{BH} shows
that both of these assertions are equivalent to: $\widetilde{p}_G$ is
Lorentzian. Now use Theorem~\ref{Tlorentz}.
\end{proof}

\begin{remark}\label{Rmetric}
As a concluding remark concerning the results mentioned until this point,
we discuss how these results hold in greater generality. First, the
definitions of a blowup and the blowup-polynomial extend to all finite
metric spaces $(X,d)$.
Now Theorems~\ref{Tmetricmatrix}, \ref{Tstable}, and~\ref{Tlorentz},
Corollary~\ref{Clorentz}, as well as Propositions~\ref{Pcoeff}
and~\ref{Psymm} extend to arbitrary finite metric spaces, possibly with
some modifications. We refer the reader to~\cite{CK-blowup} for the
details.
\end{remark}

\section{A blowup delta-matroid for graphs, and one for trees}

In addition to being a graph invariant and a multi-affine polynomial,
$p_G$ also yields a novel delta-matroid for every graph $G$.
Delta-matroids were introduced by Bouchet~\cite{Bouchet1}, and consist of
a finite ``ground set'' $E$ and a nonempty subset of its power set
$\mathcal{F} \subset 2^E$. The elements $F$ of $\mathcal{F}$ are called
\textit{feasible subsets}, and satisfy:
(1)~$\bigcup_{F \in \mathcal{F}} F = E$;
(2)~the \textit{symmetric exchange axiom:} Given $A,B \in \mathcal{F}$
and $x \in A \Delta B$ (their symmetric difference), there exists $y \in
A \Delta B$ such that $A \Delta \{ x, y \} \in \mathcal{F}$.

Br\"and\'en has shown \cite{Branden} that the set of monomials occurring
in a real-stable multi-affine polynomial forms a delta-matroid. In
particular:

\begin{defn}
The \emph{blowup delta-matroid of $G$} is denoted by
$\mathcal{M}_{M_G}$; it has ground set $V$ and feasible subsets
corresponding to the nonzero monomials in $p_G$.
\end{defn}

In fact, more is true: this delta-matroid is
\textit{linear}~\cite{Bouchet2}, in that its feasible subsets are
precisely the sets of indices $I \subset \{ 1, \dots, k \}$ for which the
principal matrix $(M_G)_{I \times I}$ is nonsingular (by
Proposition~\ref{Pcoeff}(1)). This delta-matroid appears to be unexplored
in the literature, and was not known to experts.

The goal of this section is to construct another delta-matroid
$\mathcal{M}'(T)$, this time for all trees $T$. We begin by taking a
closer look at $\mathcal{M}_{M_G}$ for $G$ a ``small'' path graph $P_k =
\{ (1,2), \dots, (k-1,k) \}$. Indeed, one can verify that, for $k \leq
4$,
\begin{equation}\label{Epathmatroid}
\mathcal{M}_{M_{P_k}} = 2^{\{ 1, \dots, k \}} \setminus \left\{ \ \{ i,
i+1, i+2 \}, \ \ \{ i, i+2 \} \ : \ 1 \leq i \leq k-2 \right\}.
\end{equation}

Let us explain why the sets $\{ i, i+1, i+2 \}$ and $\{ i, i+2 \}$ are
infeasible -- i.e., why the coefficients of the monomials $n_i n_{i+1}
n_{i+2}, n_i n_{i+2}$ in $p_{P_k}$ vanish -- for all $k \geq 3$. This
happens because the points $\{ i, i+2 \}$ are part of a graph $\{ i, i+1,
i+2 \} \cong P_3$, which is a blowup of $K_2 = P_2$ -- and in this
blowup, $i, i+2$ are copies of a vertex. More generally:

\begin{proposition}\label{Pzeroterms}
Suppose $G,H$ are finite simple connected graphs, and the tuple $\bn \in
\mathbb{Z}_{>0}^{V(G)}$ is such that $G[\bn]$ is an induced subgraph of
$H$. If some $n_v \geq 2$ and $v_1, v_2 \in G[\bn]$ are copies of $v$,
then the coefficient of $\prod_{i \in I} n_i$ in $p_H(\cdot)$ is zero
whenever $\{ v_1, v_2 \} \subset I \subset V(G[\bn])$.
\end{proposition}


\begin{proof}
By Proposition~\ref{Pcoeff}(1), it suffices to show that $(M_H)_{I \times
I}$ is singular. In turn, this holds because one verifies that the rows
of $M_H$ indexed by $v_1, v_2$ are identical.
\end{proof}

As a consequence of Proposition~\ref{Pzeroterms}, the assertion preceding
it, which involved $n_i n_{i+1} n_{i+2}$, now extends to arbitrary graphs
containing two independent nodes $a,c$ with a common neighbor $b$. It
is thus natural to return to~\eqref{Epathmatroid}, and ask two things:
(a)~Does this equality hold for all $k$? (b)~Independent of~(a), is the
right-hand side also a delta-matroid?
It is also natural to ask if (c)~the converse to
Proposition~\ref{Pzeroterms} holds: namely, if a monomial does not occur
in $p_G$, does the induced subgraph on those vertices contain two copies
of a vertex inside some blowup? The next result answers these questions.

\begin{proposition}\label{Pcounter}
Notation as above.
\begin{enumerate}
\item The right-hand side of~\eqref{Epathmatroid} is a delta-matroid for
every $k$.

\item The equality in~\eqref{Epathmatroid} holds if and only if $k \leq
8$.

\item The converse to Proposition~\ref{Pzeroterms} does not hold, even
for path graphs.
\end{enumerate}
\end{proposition}

\begin{proof}
The first part is presently explained in greater generality, for all
trees. Second, the equality in~\eqref{Epathmatroid} holds for $k \leq 8$
by explicit computations (e.g., using a computer). One also computes:
$\det (M_{P_9}) = 0$. Hence by Proposition~\ref{Pcoeff}(3), the
coefficient of $n_i n_{i+1} \cdots n_{i+8}$ in $p_{P_k}(\bn)$ is zero for
all $1 \leq i \leq k-8$. It follows that the left-hand side
of~\eqref{Epathmatroid} is a strict subset of the right-hand side, for $k
\geq 9$. The third/final assertion now follows from this computation,
since $P_9$ is not the blowup of a smaller graph.
\end{proof}

We now explore if the right-hand delta-matroid in~\eqref{Epathmatroid}
can be generalized to other graphs. This indeed turns out to hold for all
trees; to describe it, recall that the \textit{Steiner tree} $T(I)$ of a
subset of vertices $I$ of a tree is the unique smallest sub-tree
containing $I$.

\begin{theorem}\label{Ttree-blowup}
Suppose $T$ is any tree, and we define a subset of vertices $I$ to be
infeasible if its Steiner tree $T(I)$ has two leaves which are in $I$ and
have the same parent. (All other subsets are feasible.) Then the set
$\mathcal{M}'(T)$ of feasible subsets is a delta-matroid.
\end{theorem}

\noindent (See \cite{CK-blowup} for the proof.) We term this
delta-matroid the \textit{tree-blowup delta-matroid} $\mathcal{M}'(T)$.
Notice by Proposition~\ref{Pcounter}(2) that $\mathcal{M}'(P_k) \neq
\mathcal{M}_{P_k}$ for $k \geq 9$, so this is not the blowup
delta-matroid of $P_k$. Moreover, $\mathcal{M}'(T)$ also appears to not
be known to experts.

Our final result answers a natural question:
\textit{Akin to the delta-matroid $\mathcal{M}_{P_k}$, can the definition
of $\mathcal{M}'(T)$ also be extended to yield a delta-matroid for every
graph?}
In this regard, a key observation is that in Theorem~\ref{Ttree-blowup},
a set of nodes $I$ is infeasible if and only if its Steiner tree $T(I)$
is a blowup of a graph with a strictly smaller vertex set. We therefore
introduce the following two possible extensions of this version of
infeasibility to general graphs, which are both natural choices:

\begin{defn}
Let $G = (V,E)$ be a finite simple connected graph. Say that a subset $I
\subset V$ is
\begin{enumerate}
\item \emph{infeasible of the first kind} if there are vertices $v_1
\neq v_2 \in I$ and a subset $I \subset \widetilde{I} \subset V$,
satisfying:
(a)~the induced subgraph $G(\widetilde{I})$ on $\widetilde{I}$ of $G$ is
connected, and
(b)~$v_1, v_2$ have the same set of neighbors in $G(\widetilde{I})$.

\item \emph{infeasible of the second kind} if there exist $v_1 \neq v_2
\in I$ and $I \subset \widetilde{I} \subset V$, with:
(a)~the induced graph $G(\widetilde{I})$ has:
$M_{G(\widetilde{I})} = (M_G)_{\widetilde{I} \times \widetilde{I}}$ and
(b)~$v_1, v_2$ have the same neighbors in~$G(\widetilde{I})$.
\end{enumerate}
Also define $\mathcal{M}'_1(G)$ (respectively, $\mathcal{M}'_2(G)$) to
comprise all subsets of $V$ that are not infeasible of the first
(respectively, second) kind.
\end{defn}

\vspace*{-5mm}

\begin{figure}[ht]\label{Fig2}
\definecolor{xdxdff}{rgb}{0.49,0.49,1}
\hspace*{11.3cm}\begin{tikzpicture}
\draw (0,1)-- (-1.5,0);
\draw (0,1)-- (1.5,0);
\draw (0,-1)-- (-1.5,0);
\draw (0,-1)-- (1.5,0);
\draw (1.5,0)-- (3,1);
\draw (1.5,0)-- (3,-1);
\draw (1.5,-2)-- (3,-1);
\draw (1.5,-2)-- (0,-1);
%
\fill [color=xdxdff] (-1.5,0) circle (1.5pt);
\draw[color=black] (-1.5,-0.25) node {$u$};
\fill [color=xdxdff] (0,-1) circle (1.5pt);
\draw[color=black] (0,-1.4) node {$w_2$};
\fill [color=xdxdff] (1.5,0) circle (1.5pt);
\draw[color=black] (1.5,-0.25) node {$z$};
\fill [color=xdxdff] (3,1) circle (1.5pt);
\draw[color=black] (3,1.4) node {$v_1$};
\fill [color=xdxdff] (3,-1) circle (1.5pt);
\draw[color=black] (3,-1.4) node {$v_2$};
\fill [color=xdxdff] (1.5,-2) circle (1.5pt);
\draw[color=black] (1.5,-2.4) node {$x$};
\fill [color=xdxdff] (0,1) circle (1.5pt);
\draw[color=black] (0,1.4) node {$w_1$};
\draw[color=black] (1,-3.1) node {\textbf{Figure 2:} The graph ${\bf
G}_\circ$};
\end{tikzpicture}
\end{figure}

\vspace*{-54mm}As an example, if $G = T$ is a tree, then one checks

\noindent that $\mathcal{M}'_1(T) = \mathcal{M}'_2(T) = \mathcal{M}'(T)$.
It is now natural to

\noindent ask if either $\mathcal{M}'_1(G)$ or $\mathcal{M}'_2(G)$ is a
delta-matroid for

\noindent all graphs $G$. It turns out that this is not the case:

\begin{proposition}[\cite{CK-blowup}]
For the graph $G = {\bf G}_\circ$ (see Figure~2),

\noindent neither $\mathcal{M}'_1(G)$ nor $\mathcal{M}'_2(G)$ is a
delta-matroid.
\end{proposition}

In closing, we note the above results describe several

\noindent novel invariants associated to finite simple connected

\noindent graphs (in fact, finite metric spaces). These include the

\noindent polynomials $p_G(\bn)$, $u_G(n)$; the delta-matroid
$\mathcal{M}_{M_G}$ (and 

\noindent $\mathcal{M}'(G)$ for $G$ a tree); but also ``simpler''
invariants like $\deg p_G$, $\deg u_G$. (These degrees are not
necessarily $|V|$ even if $G$ is not a blowup of a smaller graph; e.g.,
$G = P_k$ for $k \geq 9$, by Proposition~\ref{Pcounter}.) It would be
desirable and interesting to explore if these are relatable to more
``familiar'' combinatorial graph invariants.





\begin{thebibliography}{88}
\bibitem{AGV}
N.~Anari, S.~Oveis Gharan, and C.~Vinzant.
\newblock Log-concave polynomials, entropy, and a deterministic
approximation algorithm for counting bases of matroids.
\newblock \href{http://dx.doi.org/10.1109/FOCS.2018.00013}{\em FOCS 2018
Proc.} (59th Annual IEEE Symposium), pp.\ 35--46, 2018.

\bibitem{AH}
M.~Aouchiche and P.~Hansen.
\newblock Distance spectra of graphs: a survey.
\newblock \href{http://dx.doi.org/10.1016/j.laa.2014.06.010}{\em 
Linear Algebra Appl.}, 458:301--386, 2014.

\bibitem{BB1}
J.~Borcea and P.~Br\"and\'en.
\newblock Applications of stable polynomials to mixed determinants:
Johnson's conjectures, unimodality, and symmetrized Fischer products.
\newblock \href{http://dx.doi.org/10.1215/00127094-2008-018}{\em Duke
Math.\ J.}, 143(2):205--223, 2008.

\bibitem{BB2}
J.~Borcea and P.~Br\"and\'en.
\newblock P\'olya--Schur master theorems for circular domains and their
boundaries.
\newblock \href{http://dx.doi.org/10.4007/annals.2009.170.465}{\em Ann.\
of Math.\ (2)}, 170(1):465--492, 2009.

\bibitem{BB3}
J.~Borcea and P.~Br\"and\'en.
\newblock The Lee--Yang and P\'olya--Schur programs.
I. Linear operators preserving stability.
\newblock \href{http://dx.doi.org/10.1007/s00222-009-0189-3}{\em Invent.\
Math.}, 177(3):541--569, 2009.

\bibitem{BBL}
J.~Borcea, P.~Br\"and\'en, and T.M.~Liggett.
\newblock Negative dependence and the geometry of polynomials.
\newblock \href{http://dx.doi.org/10.1090/S0894-0347-08-00618-8}{\em J.\
Amer.\ Math.\ Soc.}, 22(2):521--567, 2009.

\bibitem{Bouchet1}
A.~Bouchet.
\newblock Greedy algorithms and symmetric matroids.
\newblock {\em Math Progr.}, 38:147--159, 1987.

\bibitem{Bouchet2}
A.~Bouchet.
\newblock Representability of $\Delta$-matroids.
\newblock In: {\em Proc.\ 6th Hungarian Coll.\ Combin.\ 1987}, Colloq.\
Math.\ Soc.\ J\'anos Bolyai, 52:167--182, 1988.

\bibitem{Branden}
P.~Br\"and\'en.
\newblock Polynomials with the half-plane property and matroid theory.
\newblock \href{http://dx.doi.org/10.1016/j.aim.2007.05.011}{\em Adv.\ in
Math.}, 216(1):302--320, 2007.

\bibitem{Branden-survey}
P.~Br\"and\'en.
\newblock Unimodality, log-concavity, real-rootedness and beyond.
\newblock \href{http://dx.doi.org/10.1201/b18255}{\em Handbook of
Enumerative Combinatorics} (M.\ Bona, Ed.), 437--483, Boca Raton, 2015.

\bibitem{BH}
P.~Br\"and\'en and J.~Huh.
\newblock Lorentzian polynomials.
\newblock \href{http://dx.doi.org/10.4007/annals.2020.192.3.4}{\em Ann.\
of Math.}, 192(3):821--891, 2020.

\bibitem{Brenti}
F.~Brenti.
\newblock {\em Unimodal, log-concave, and P\'olya frequency sequences
in combinatorics}.
\newblock Mem.\ Amer.\ Math.\ Soc., vol.~413, American Mathematical
Society, Providence, 1989.

\bibitem{COSW}
Y.-B.~Choe, J.G.~Oxley, A.D.~Sokal, and D.G.~Wagner.
\newblock Homogeneous multivariate polynomials with the half-plane
property.
\newblock \href{http://dx.doi.org/10.1016/S0196-8858(03)00078-2}{\em
Adv.\ in Appl.\ Math.}, 32(1--2):88--187, 2004.

\bibitem{CK-tree}
P.N.~Choudhury and A.~Khare.
\newblock Distance matrices of a tree: two more invariants, and in a
unified framework.
\newblock {\em Preprint},
\href{http://arxiv.org/abs/1903.11566}{arXiv:1903.11566}, 2019.

\bibitem{CK-blowup}
P.N.~Choudhury and A.~Khare.
\newblock The blowup-polynomial of a metric space: connections to stable
polynomials, graphs and their distance spectra.
\newblock {\em Preprint},
\href{http://arxiv.org/abs/2105.12111}{arxiv.org/2105.12111}, 2021.

\bibitem{GL}
R.L.~Graham and L.~Lov\'asz.
\newblock Distance matrix polynomials of trees.
\newblock \href{http://dx./doi.org/10.1016/0001-8708(78)90005-1}{\em
Adv.\ in Math.}, 29(1):60--88, 1978.

\bibitem{Graham-Pollak}
R.L.~Graham and H.O.~Pollak.
\newblock On the addressing problem for loop switching.
\newblock {\em Bell System Tech.\ J.}, 50:2495--2519, 1971.

\bibitem{GKR-critG}
D.~Guillot, A.~Khare, and B.~Rajaratnam.
\newblock Critical exponents of graphs.
\newblock \href{http://dx.doi.org/10.1016/j.jcta.2015.11.003}%
{\em J.\ Combin.\ Th.\ Ser.\ A}, 139:30--58, 2016.

\bibitem{Gurvits}
L.~Gurvits.
\newblock On multivariate Newton-like inequalities.
\newblock \href{http://dx.doi.org/10.1007/978-3-642-03562-3_4}{\em Adv.\
in Combin.\ Math.}, Springer, Berlin, pp.\ 61--78, 2009.

\bibitem{HHN}
H.~Hatami, J.~Hirst, and S.~Norine.
\newblock The inducibility of blow-up graphs.
\newblock \href{http://dx.doi.org/10.1016/j.jctb.2014.06.005}{\em J.\
Combin.\ Th.\ Ser.\ B}, 109:196--212, 2014.

\bibitem{KKOT}
J.~Kim, D.~K\"{u}hn, D.~Osthus, and M.~Tyomkyn.
\newblock A blow-up lemma for approximate decompositions.
\newblock \href{http://dx.doi.org/10.1090/tran/7411}{\em Trans.\ Amer.\
Math.\ Soc.}, 371(7):4655--4742, 2019.

\bibitem{KSS}
J.~Koml\'os, G.N.~S\'ark\"ozy, and E.~Szemer\'edi.
\newblock Blow-up lemma.
\newblock \href{http://dx.doi.org/10.1007/BF01196135}{\em Combinatorica},
17:109--123, 1997.

\bibitem{MSS1}
A.W.~Marcus, D.A.~Spielman, and N.~Srivastava.
\newblock Interlacing families I:
Bipartite Ramanujan graphs of all degrees.
\newblock \href{http://dx.doi.org/10.4007/annals.2015.182.1.7}{\em Ann.\
of Math.\ (2)}, 182(1):307--325, 2015.

\bibitem{MSS2}
A.W.~Marcus, D.A.~Spielman, and N.~Srivastava.
\newblock Interlacing families II.
Mixed characteristic polynomials and the Kadison-Singer problem.
\newblock \href{http://dx.doi.org/10.4007/annals.2015.182.1.8}{\em Ann.\
of Math.\ (2)}, 182(1):327--350, 2015.

\bibitem{Pemantle}
R.~Pemantle.
\newblock Hyperbolicity and stable polynomials in combinatorics and
probability.
\newblock \href{http://dx.doi.org/10.4310/CDM.2011.v2011.n1.a2}{\em
Current Develop.\ Math.}, 2011:57--123, International Press, Boston, 2012.

\bibitem{Stanley}
R.P.~Stanley.
\newblock Log-concave and unimodal sequences in algebra, combinatorics,
and geometry.
\newblock In: \emph{Graph theory and its applications: East and West
(Jinan, 1986)}, vol.\ 576 of Ann.\ New York Acad.\ Sci., pp.\ 500--535.
New York Acad.\ Sci., New York, 1989.

\bibitem{WW}
D.G.~Wagner and Y.~Wei.
\newblock A criterion for the half-plane property.
\newblock \href{http://dx.doi.org/10.1016/j.disc.2008.02.005}{\em
Discrete Math.}, 309(6):1385--1390, 2009.
\end{thebibliography}


\end{document}